\documentclass[reqno,a4paper]{amsart}
\usepackage[UKenglish]{babel}

\usepackage[monochrome]{xcolor} 

\usepackage{amsmath, amsthm, amssymb}
\usepackage[foot]{amsaddr}

\usepackage{enumerate}
\usepackage{algorithmic,algorithm}
\usepackage{hyperref}
\usepackage{graphicx}

\theoremstyle{theorem}
\newtheorem{theorem}{Theorem}[section]
\newtheorem{lemma}[theorem]{Lemma}
\newtheorem{corollary}[theorem]{Corollary}

\theoremstyle{definition}
\newtheorem{assumption}[theorem]{Assumption}

\newcommand{\E}{\ensuremath{\mathbb{E}}}  
\newcommand{\F}{\ensuremath{\mathbb{F}}}  
\newcommand{\R}{\ensuremath{\mathbb{R}}}

\newcommand{\Prob}{\ensuremath{\mathbb{P}}}

\title{A modified MSA for stochastic control problems}
\ifpdf
\hypersetup{
  pdftitle={A modified MSA for stochastic control problems},
  pdfauthor={B. Kerimkulov, D. \v{S}i\v{s}ka, and L. Szpruch}
}
\fi

\date{\today, }
\author{B. Kerimkulov$^{1,2}$}
\address{$^1$\href{http://www.maxwell.ac.uk/migsaa}{Maxwell Institute Graduate School in Analysis and Applications}}
\address{$^2$\href{https://www.maths.ed.ac.uk}{School of Mathematics, University of Edinburgh}}
\email{B.Kerimkulov@sms.ed.ac.uk}
\author{D. \v{S}i\v{s}ka$^{2,3}$}
\address{$^3$\href{https://vegaprotocol.io}{Vega Protocol}}
\email{D.Siska@ed.ac.uk}
\author{{\L}. Szpruch$^{2,4}$}
\address{$^4$\href{https://www.turing.ac.uk}{Alan Turing Institute}}
\email{L.Szpruch@ed.ac.uk}

\thanks{Supported by the Alan Turing Institute under EPSRC grant no. EP/N510129/1 and by The Maxwell Institute Graduate School in Analysis and its Applications, a Centre for Doctoral Training funded by the UK Engineering and Physical Sciences Research Council (grant EP/L016508/01), the Scottish Funding Council, Heriot-Watt University and the University of Edinburgh.}


\subjclass[2010]{}

\begin{document}
\maketitle

\begin{abstract}
The classical Method of Successive Approximations (MSA) is an iterative method for solving stochastic control problems and is derived from Pontryagin's optimality principle.
It is known that the MSA may fail to converge. 
Using careful estimates for the backward stochastic differential equation (BSDE) this paper suggests a modification to the MSA algorithm. 
This modified MSA is shown to converge for general stochastic control problems with control in both the drift and diffusion coefficients. 
Under some additional assumptions the rate of convergence is shown.
The results are valid without restrictions on the time horizon of the control problem, in contrast to iterative methods based on the theory of forward-backward stochastic differential equations. 
\end{abstract}

\section{Introduction}
Stochastic control problems appear naturally in a range of applications in engineering, economics and finance. 
With the exception of very specific cases such as the linear-quadratic control problem in engineering or Merton 
portfolio optimization task in finance,  stochastic control problems typically have no closed form 
solutions and have to be solved numerically. In this work, we consider a modification to the method of successive approximations (MSA), see Algorithm~\ref{alg mmsa}. 
The MSA is essentially a way of applying the Pontryagin's optimality principle to get numerical solutions of stochastic control problems.

We will consider the continuous space, continuous time problem where 
the controlled system is modelled by an $\mathbb R^d$-valued diffusion process.
Let $W$ be a $d'$-dimensional Wiener martingale on a filtered probability space $(\Omega, \mathcal{F}, (\mathcal{F}_t)_{t\ge 0}, \Prob)$. 
We will provide exact assumptions we need in Section~\ref{sec main}.
For now, let us fix a finite time $T\in(0,\infty)$ and consider the controlled stochastic differential equation (SDE) for given measurable functions
$b:[0,T]\times\R^d\times A\to \R^d$ and $\sigma:[0,T]\times\R^d\times A\rightarrow\R^{d\times d'}$
\begin{equation}\label{sde}
dX_s=b(s,X_s,\alpha_s)\,ds+\sigma(s,X_s,\alpha_s)\,dW_s\,,\,\, s\in [0,T]\,,\,\,\, X_0 = x\,.
\end{equation}
Here $\alpha=(\alpha_s)_{s\in[0,T]}$ is a control process belonging to the space of admissible controls
$\mathcal A$, valued in a separable metric space $A$ and we will write $X^{\alpha}$ to denote the unique solution of~\eqref{sde} 
which starts from $x$ at time $0$ whilst being controlled by $\alpha$.
Furthermore let $f:[0,T]\times \mathbb R^d \times A \to \mathbb R$ and $g:\mathbb R^d \to \mathbb R$ be given measurable functions and consider the gain functional
\begin{equation}\label{problem}
J(x,\alpha):=\E\left[\int_{0}^{T}f(s,X_s^{\alpha},\alpha_s)ds+g(X_T^{\alpha})\right]
\end{equation}
for all $x\in\R^d$ and $\alpha\in\mathcal{A}$. 
We want to solve the optimisation problem i.e. to find the optimal control $\alpha^*$ which achieves the minimum of~\eqref{problem} (or, if the infimum
cannot be reached by $\alpha \in \mathcal A$ then an $\varepsilon$-optimal control $\alpha^\varepsilon \in \mathcal A$ such that $\inf_{\alpha\in\mathcal{A}}J(x,\alpha) \leq J(x,\alpha^\varepsilon) + \varepsilon$). 

In the present paper, we study an approach based on Pontryagin's optimality principle, see e.g.~\cite{boltyanskii},~\cite{pontryagin} or~\cite{carmona book}. The main idea is to consider optimality conditions for controls of the problem~\eqref{problem}. Given $b, \sigma$ and $f$ we define the Hamiltonian $\mathcal{H}:[0,T]\times\R^d\times\R^d\times\R^{d\times d'}\times A\rightarrow \R$ as
\begin{equation}\label{eq: hamiltonian}
\mathcal{H}(t,x,y,z,a)=b(t,x,a)\cdot y+\text{tr}(\sigma^\top (t,x,a)z)+f(t,x,a)\,.
\end{equation}
Consider for each $\alpha\in\mathcal{A}$, the BSDE, called the adjoint equation
\begin{equation}\label{eq: adjoint equation}
dY_s^{\alpha}=-D_x\mathcal{H}(s,X_s^{\alpha},Y_s^{\alpha},Z_s^{\alpha},\alpha_s)\,ds+Z_s^{\alpha}\,dW_s,\,\,\,Y_T^{\alpha}=D_xg(X_T^{\alpha}),\,\,\,s\in[0,T]\,.
\end{equation}
It is well known from Pontryagin's optimality principle that, if an admissible control $\alpha^*\in\mathcal{A}$ is optimal, $X^{\alpha^*}$ is the corresponding optimally controlled dynamic~\eqref{sde} and $(Y^{\alpha^*},Z^{\alpha^*})$ is the solution to the associated adjoint equation~\eqref{eq: adjoint equation}, then $\forall a\in A$ and $\forall s\in[0,T]$ the following holds
\begin{equation}
\label{eq hamlitonian pointwise ineq}
\mathcal{H}(s,X_s^{\alpha^*},Y_s^{\alpha^*},Z_s^{\alpha^*},\alpha_s^*)\le \mathcal{H}(s,X_s^{\alpha^*},Y_s^{\alpha^*},Z_s^{\alpha^*},a)\,\,\,\text{a.s.}
\end{equation}

We now define the augmented Hamiltonian $\mathcal{\tilde{H}}:[0,T]\times\R^d\times\R^d\times\R^{d\times d'}\times A\times A\rightarrow \R$ for some $\rho\ge0$ by
\begin{equation}\label{eq: augmented hamiltonian}
\begin{split}
\mathcal{\tilde{H}}&(t,x,y,z,a',a):=\mathcal{H}(t,x,y,z,a)+\frac{1}{2}\rho|b(t,x,a)-b(t,x,a')|^2\\
&\frac{1}{2}\rho|\sigma(t,x,a)-\sigma(t,x,a')|^2+\frac{1}{2}\rho\left|D_x\mathcal{H}(t,x,y,z,a)-D_x\mathcal{H}(t,x,y,z,a')\right|^2\,.
\end{split}
\end{equation}
Notice that when $\rho=0$ we have exactly the definition of Hamiltonian~\eqref{eq: hamiltonian}. 
Given the augmented Hamiltonian, let us introduce the modified MSA in Algorithm~\ref{alg mmsa} which consists of successive integrations of the state and adjoint systems and updates to the control.
Notice that the backward SDE depends on the Hamiltonian $\mathcal{H}$, while the control update step comes from minimizing the augmented Hamiltonian $\tilde{\mathcal{H}}$.

\begin{algorithm}[h!]
	\caption{Modified Method of Successive Approximations:}
	\label{alg mmsa}
	\begin{algorithmic}
		\STATE{Initialisation: make a guess of the control $\alpha^0 = (\alpha^0_s)_{s\in[0,T]}$.}
		\WHILE{difference between $J(x,\alpha^n)$ and $J(x,\alpha^{n-1})$ is large}
		\STATE{Given a control $\alpha^{n-1} = (\alpha^{n-1}_s)_{s\in[0,T]}$ solve the following forward SDE, then solve backward SDE:
			\begin{equation}
			\label{eq MMSA FBSDE}
			\begin{split}
			dX^{n}_s&=b(s,X^{n}_s,\alpha^{n-1}_s)\,ds+\sigma(s,X^{n}_s,\alpha^{n-1}_s)\,dW_s\,,\,\,\,\, X^{n}_0 = x\,,\\
			dY^{n}_s&=-D_x\mathcal{H}(s,X^{n}_s,Y^{n}_s,Z^{n}_s,\alpha^{n-1}_s)\,ds+Z^{n}_s\,dW_s,\,\,\,Y^{n}_T=D_xg(X^{n}_T)\,.
			\end{split}
			\end{equation}
		}
		\STATE{Update the control
			\begin{equation}
			\label{eq def anplus1 MMSA}
			\alpha^{n}_s \in \arg \min_{a\in A}\tilde{\mathcal{H}}(s,X^{n}_s,Y^{n}_s,Z^{n}_s,\alpha^{n-1}_s,a)\,,\,\, \forall s\in[0,T]\,.	
			\end{equation}
		}
		\ENDWHILE
		\RETURN $\alpha^{n}$.
	\end{algorithmic}
\end{algorithm}
The method of successive approximations (i.e. case $\rho=0$) for numerical solution of deterministic control problems was proposed already in~\cite{chernousko}. 
Recent application of the modified MSA to a deep learning problem has been studied in~\cite{weinan e}, where they formulated the training of deep neural networks as an optimal control problem and introduced the modified method of successive approximations as an alternative training algorithm for deep learning. 
For us, the main motivation to explore the modified MSA for stochastic control problems is to obtain convergence, ideally with rate, of an iterative algorithm, applicable to problems with the control in the diffusion part of the controlled dynamics. 
This is in contrast to~\cite{kerimkulov siska szpruch} where convergence rate of an the Bellman--Howard policy iteration is shown but only for control problems with no control in the diffusion part of the controlled dynamics.

In Lemma~\ref{lem estimate for difference of J}, which can be established using careful BSDE estimates, we can see the estimate on the change of $J$  when we do a minimization step of Hamiltonian as in~\eqref{eq def anplus1 MMSA}. 
If the sum of the last three terms of~\eqref{eq estimate for difference of J} is bigger than the first term, then for classical MSA algorithm (i.e. case $\rho=0$) we cannot guarantee that we do an update of the control in optimal descent direction of $J$. 
That means that the method of successive approximations may diverge. 
To overcome this, we need to modify the algorithm in such way so that we ensure convergence. 
With this in mind the desirability of the the augmented Hamiltonian~\eqref{eq: augmented hamiltonian} for updating the control becomes clear, as long as it still characterises optimal controls like $\mathcal H$ does. 
Theorem~\ref{thr extended pop} answers this question affirmatively which opens the way to the modified MSA.
In Theorem~\ref{thr convergence of modified MSA} we show that the modified method of successive approximations,
converges for arbitrary $T$, and in Corollary~\ref{cor mmsa rate}, we show {\color{red}logarithmic} convergence rate for certain stochastic control problems.

We observe that the forward and backward dynamics in~\eqref{eq MMSA FBSDE} are decoupled, due to the iteration used. 
Therefore, it can be efficiently approximated, even in high dimension, using deep learning methods, see~\cite{han} and~\cite{sabate}. 
However, the minimization step~\eqref{eq def anplus1 MMSA} might be computationally expensive for some problems. A possible approach circumventing this is to replace the full minimization of~\eqref{eq def anplus1 MMSA} by gradient descent. A continuous version of this gradient flow is analyzed in~\cite{siska szpruch}.

The main contributions of this paper are the probabilistic proof of convergence of the modified method of successive approximations and establishing convergence rate for a specific type of optimal control problems.

This paper is organised as follow: in Section~\ref{sec related work} we compare our results with existing work. 
In Section~\ref{sec main} we state the assumptions and main results. 
In Section~\ref{sec proofs} we collect all proofs.
Finally, in Appendix~\ref{sec appendix} we recall an auxiliary lemma which is needed in the proof of Corollary~\ref{cor mmsa rate}.

\subsection{Related work}\label{sec related work} One can solve the stochastic optimal control problem using dynamic programming principle. 
It is well known, see e.g. Krylov~\cite{krylov controlled}, that under reasonable assumptions the value function, defined as infimum of~\eqref{problem} over all admissible controls, satisfies the Bellman partial differential equation (PDE).
There are several approaches to solve this nonlinear problem. One may apply a finite difference method to discretise the Bellman PDE and get a high dimensional nonlinear system of equations, see e.g~\cite{dong krylov rate} or~\cite{gyongy siska finite}. Or one may linearize the Bellman PDE and then iterate. The classical approach is the Bellman-Howard policy improvement / iteration algorithm, see e.g.~\cite{Bellman},~\cite{Bellman:1957} or~\cite{howard:dp}. The algorithm is initialised with a ``guess" of Markovian control. 
Given a Markovian control strategy at step $n$ one solves a linear PDE with the given control fixed and then one uses the solution to the linear PDE to update the Markovian control, see e.g.~\cite{jacka miatovic policy},~\cite{jacka miatovic siraj coupling} or ~\cite{maeda jacka evaluation}. 
In~\cite{kerimkulov siska szpruch}, a global rate of convergence and stability for the policy iteration algorithm has been established using backward stochastic differential equations (BSDE) theory. However, the result only applies to stochastic control problems with no control in the diffusion coefficient of the controlled dynamics. 

It is known that the solution of the stochastic optimal control problem can be obtained from a corresponding forward backward stochastic differential equation (FBSDE) via the stochastic optimality principle, see~\cite[Chapter 8.1]{zhang book}. Indeed, let us consider~\eqref{sde} and~\eqref{eq: adjoint equation}, and recall from the stochastic optimality principle, see~\cite[Theorem 4.12]{carmona book}, that for the optimal control $\alpha^*=(\alpha_s^*)_{s\in[0,T]}$ we have that~\eqref{eq hamlitonian pointwise ineq} holds.
Assume that under some conditions on $b,\sigma$ and $f$ we have that the first order condition stated above uniquely determines $\alpha^*$ for $s\in[0,T]$ by
\begin{equation}\label{eq alpha from phi}
\alpha_s^*=\varphi(s,X_s^{\alpha^*},Y_s^{\alpha^*},Z_s^{\alpha^*})\,,
\end{equation}
for some function $\varphi$. Therefore, after plugging~\eqref{eq alpha from phi} into~\eqref{sde} and~\eqref{eq: adjoint equation}, we obtain the following coupled FBSDE:
\begin{equation}\label{eq FBSDE from control problem}
\begin{split}
dX_s^{\alpha^*}&=\bar{b}(s,X_s^{\alpha^*},Y_s^{\alpha^*},Z_s^{\alpha^*})\,ds+\bar{\sigma}(s,X_s^{\alpha^*},Y_s^{\alpha^*},Z_s^{\alpha^*})\,dW_s\,,\,\, s\in [0,T]\,,\,\,\, X_0^{\alpha^*} = x\,.\\
dY_s^{\alpha^*}&=-D_x\bar{\mathcal{H}}(s,X_s^{\alpha^*},Y_s^{\alpha^*},Z_s^{\alpha^*})\,ds+Z_s^{\alpha^*}\,dW_s,\,\,\,Y_T=D_xg(X_T^{\alpha^*}),\,\,\,s\in[0,T]\,,
\end{split}
\end{equation}
where $(\bar{b},\bar{\sigma})(s,X_s^{\alpha^*},Y_s^{\alpha^*},Z_s^{\alpha^*})=(b,\sigma)(s,X_s^{\alpha^*},\varphi(s,X_s^{\alpha^*},Y_s^{\alpha^*},Z_s^{\alpha^*}))$ and\\ $\bar{\mathcal{H}}(s,X_s^{\alpha^*},Y_s^{\alpha^*},Z_s^{\alpha^*})=\mathcal{H}(s,X_s^{\alpha^*},Y_s^{\alpha^*},Z_s^{\alpha^*},\varphi(s,X_s^{\alpha^*},Y_s^{\alpha^*},Z_s^{\alpha^*}))$. It is worth mentioning that when $\sigma$ does not depend on the control $\bar{\sigma}$ will depend on forward process and time only. This means that $\bar{\sigma}$ does not have $Y$ and $Z$ components.

The theory of FBSDE has been studied widely and there are several methods to show the existence and uniqueness result, and a number of numerical algorithms have been proposed based on those methods. First is the method of contraction mapping.
It was first studied by Antonelli~\cite{antonelli} and later by Pardoux and Tang~\cite{pardoux}. The main idea there is to show that a certain map is a contraction, and then to apply a fixed point argument. However, it turns out that this method works only for small enough time horizon $T$. In the case when $\bar{\sigma}$ does not depend on $Y$ and $Z$, having small $T$ is sufficient to get contraction. Otherwise, one needs to assume additionally that the Lipschitz constants of $\bar{\sigma}$ in $z$ and that of $g$ in $x$ satisfy a certain inequality, see~\cite[Theorem 8.2.1]{zhang book}. Using the method of contraction mapping one can then implement a Picard-iteration-type numerical algorithm and show exponential convergence for small $T$. The second method is the Four Step Scheme. It was introduced by Ma, Protter and Yong, see~\cite{ma}, and was later studied by Delarue~\cite{delarue}. The idea is to use a decoupling function and then study an associated quasi-linear PDE.  We note that in~\cite{ma,delarue} the forward diffusion coefficient $\bar{\sigma}$ does not depend on $Z$. This corresponds to stochastic control problems with the uncontrolled diffusion coefficient. The numerical algorithms based on this method exploits the numerical solution of the associated quasi-linear PDE and therefore faces some limitations for high dimensional problems, see Douglas, Ma and Protter~\cite{douglas}, Milstein
and Tretyakov~\cite{milstein}, Ma, Shen and Zhao~\cite{shen} and Delarue and Menozzi~\cite{menozzi}. Guo, Zhang and Zhuo~\cite{guo} proposed a numerical scheme for high-dimensional quasi-linear PDE associated with the coupled FBSDE when $\bar{\sigma}$ does not depend on $Z$, which is based on a monotone scheme and on probabilistic approach. Finally, there is the method of continuation. This method was developed by Hu and Peng~\cite{hu},  Peng and Wu~\cite{peng} and  by Yong~\cite{yong}. It allows them to show the existence and uniqueness result for arbitrary $T$ under monotonicity conditions on the coefficients, which one would not expect to apply to FBSDEs arising from a control problem as described by~\eqref{eq alpha from phi},~\eqref{eq FBSDE from control problem}. Recently, deep learning methods have been applied to solving FBSDEs. In~\cite{shaolin}, three algorithms for solving fully coupled FBSDEs which have good accuracy and performance for high-dimensional problems are provided. One of the algorithms is based on the Picard iteration and it converges, but only for small enough $T$. 
{\color{red}Such method for solving high-dimensional FBSDEs has also been proposed in~\cite{han long}}.

\section{Main results}
\label{sec main}
We fix a finite horizon $T\in (0,\infty)$.
Let $A$ be a separable metric space.
This is the space where the control processes $\alpha$ take values. 
We fix a filtered probability space $(\Omega, \mathcal{F}, \F=(\mathcal{F}_t)_{0\le t\le T}, \Prob)$.
Let 
$W=(W_t)_{t\in [0,T]}$	
be a $d'$-dimensional Wiener martingale on this space. {\color{red}
By $\E_t$ we denote the conditional expectation with respect to $\mathcal{F}_t$. Let $|\cdot|$ denote any norm in a finite dimensional Euclidean space. By $\|\cdot\|_{L^\infty}$ we denote the norm in $L^{\infty}(\Omega)$. 
Let $\|Z\|_{\mathbb{H}^{\infty}}:=\text{ess}\sup_{(t,\omega)}|Z_t(\omega)|$ for any predictable process $Z$. We understand the following as $D_x\sigma=D_{x_l}\sigma^{ij}$,  $D_x^2b=D^2_{x_lx_n}b^i$ and $D_x^2\sigma=D^2_{x_lx_n}\sigma^{ij}$, where $i,l,n=1,2,\dots,d$ and $j=1,2,\dots,d'$. By $Z^\top$ we denote the transpose of $Z$.}
The state of the system is governed by the controlled SDE~\eqref{sde}\,. The corresponding adjoint equation satisfies~\eqref{eq: adjoint equation}.
\begin{assumption}
	\label{assumption controlled SDE for modified MSA}
	The functions $b$ and $\sigma$ are jointly continuous in $t$ and twice differentiable in $x$.
	There exists $K\ge 0$ 
	such that $\forall x\in\R^d,\forall a\in A,\forall t\in[0,T]$,
	\begin{equation}\label{eq derivative bounds}
	|D_x b(t,x,a)|+|D_x\sigma(t,x,a)|+|D^2_x b(t,x,a)|\le K\,.
	\end{equation}
	{\color{red}Moreover, assume that $D_x^2\sigma(t,x,a)=0$ $\forall x\in\R^d,\forall a\in A,\forall t\in[0,T]$.}
\end{assumption}
{\color{red} 
Clearly the assumption~\eqref{eq lipschitz cont} implies that $\forall x,x'\in\R^d,\forall a\in A,\forall t\in[0,T]$ we have
\begin{equation}\label{eq lipschitz cont}
|b(t,x,a)-b(t,x',a)|+|\sigma(t,x,a)-\sigma(t,x',a)|\le K|x-x'|\,.
\end{equation} 
The assumption that $D_x^2\sigma(t,x,a)=0$ $\forall x\in\R^d,\forall a\in A,\forall t\in[0,T]$ is needed so that~\eqref{eq D_x^2 H is bounded}, in the proof of Lemma~\ref{lem Y Z bounded}, holds. 
Without this assumption~\eqref{eq D_x^2 H is bounded} would only hold if we could show that $\|Z^\alpha\|_{\mathbb H^\infty} < \infty$.  
Without additional regularization of the control problem this is impossible. 
Indeed, with~\cite[Proposition 5.3]{elkaroui peng quenez backward} we see that $Z^\alpha_t$ is a version of $D_t Y_t^\alpha$ (the Malliavin derivative of $Y_t^\alpha$) and $D_t Y_t^\alpha$ itself satisfies an a linear BSDE. 
However, to obtain the estimates using this representation, one term that arises is $D_t \alpha_s$ where $t\in [0,T]$ and $s\in [t,T]$.
So we would need $\text{ess}\sup_{\omega\in \Omega, t\in (0,T),s\in(t,T)}|D_t \alpha_s(\omega)|<\infty $.
This is not necessarily the case here. 
}

\begin{assumption}\label{ass second order derivative of Hamiltonian bounded}
	The functions $f$ is joinly continuous in $t$, and $f$ and $\sigma$ are  twice differentiable in $x$.
	There is a constant $K\ge0$ such that  $\forall x,\forall a\in A,\forall t\in[0,T]$
	\begin{equation}
	|D_x g(x)|+|D_xf(t,x,a)|+|D^2_x g(x)|+|D_x^2f(t,x,a)|\le K\,.
	\end{equation}
\end{assumption}
Under these assumptions, we can obtain the following estimate.	
\begin{lemma}\label{lem estimate for difference of J}
	Let Assumption~\ref{assumption controlled SDE for modified MSA} and~\ref{ass second order derivative of Hamiltonian bounded} hold. Then for any admissible controls $\varphi$ and $\theta$ there exists a constant $C>0$ such that
	\begin{equation}\label{eq estimate for difference of J}
	\begin{split}
	J(x,\varphi)&-J(x,\theta)\le\E\int_0^T[\mathcal{H}(s,X^\theta_s,Y^\theta_s,Z^\theta_s,\varphi_s)-\mathcal{H}(s,X^\theta_s,Y^\theta_s,Z^\theta_s,\theta_s)]\,ds\\
	&+C\E\int_0^T|b(s,X^\theta_s,\varphi_s)-b(s,X^\theta_s,\theta_s)|^2\,ds\\
	&+C\E\int_0^T|\sigma(s,X^\theta_s,\varphi_s)-\sigma(s,X^\theta_s,\theta_s)|^2\,ds\\
	&+C\E\int_0^T|D_x\mathcal{H}(s,X^\theta_s,Y^\theta_s,Z^\theta_s,\varphi_s)-D_x\mathcal{H}(s,X^\theta_s,Y^\theta_s,Z^\theta_s,\theta_s)|^2\,ds\,.
	\end{split}
	\end{equation}
\end{lemma}
The proof will be given in Section~\ref{sec proofs}. We now state a necessary condition for optimality for the augmented Hamiltonian.
\begin{theorem}[Extended Pontryagin's optimality principle]\label{thr extended pop}
	Let $\alpha^*$ be the (locally) optimal control, $X^{\alpha^*}$ be the associated controlled state solving~\eqref{sde}, and $(Y^{\alpha^*},Z^{\alpha^*})$ be the associated adjoint processes solving~\eqref{eq: adjoint equation}. Then for any $a\in A$ we have
	\begin{equation}\label{eq extended pontryagin}
	\tilde{\mathcal{H}}(s,X_s^{\alpha^*},Y_s^{\alpha^*},Z_s^{\alpha^*},\alpha_s^*,\alpha_s^*)\le \tilde{\mathcal{H}}(t,X_s^{\alpha^*},Y_s^{\alpha^*},Z_s^{\alpha^*},\alpha_s^*,a)\,,\,\,\,\forall s\in[0,T]\,.
	\end{equation}
\end{theorem}
The proof of Theorem~\ref{thr extended pop} will come in Section~\ref{sec proofs}. We are now ready to present the main result of the paper.
\begin{theorem}\label{thr convergence of modified MSA}
	Let Assumptions~\ref{assumption controlled SDE for modified MSA} and~\ref{ass second order derivative of Hamiltonian bounded} hold. Then Algorithm~\ref{alg mmsa} converges to {\color{red}a local minimum of~\eqref{problem} for sufficiently large $\rho>0$}. 
\end{theorem}
Theorem~\ref{thr convergence of modified MSA} will be proved in Section~\ref{sec proofs}. 
{\color{red}It can be seen from the proof that $\rho$ needs to be two times larger than the constant appearing in Lemma~\ref{lem estimate for difference of J}, which itself depends increases with $T, d$ and constants from Assumption~\ref{assumption controlled SDE for modified MSA} and~\ref{ass second order derivative of Hamiltonian bounded}.} 

{\color{red}We cannot guarantee that the Algorithm~\ref{alg mmsa} converges
to the optimal control which minimizes~\eqref{problem}, since the extended Pontryagin's optimality principle, see Theorem~\ref{thr extended pop}, is the necessary condition for optimality.
The sufficient condition for optimality tells us that to get the optimal control we
need to assume convexity of the Hamiltonian in state and control variables, and need to assume convexity of the terminal cost function. To that end, we need to assume convexity of $b,\sigma, f$ and $g$ in $x$ and $a$.
}

In the following corollary, we show that under a particular setting of the problem we have {\color{red}logarithmic} convergence of the modified method of successive approximations to the true solution of the problem.

\begin{corollary}\label{cor mmsa rate}
	Let Assumptions~\ref{assumption controlled SDE for modified MSA} and~\ref{ass second order derivative of Hamiltonian bounded} hold. Moreover, if $b,\sigma$ and $f$ are in the form of
	\begin{equation*}
	\begin{split}
	&b(t,x,a)=b_1(t)x+b_2(t,a)\,,\\
	&\sigma(t,x,a)=\sigma_1(t)x+\sigma_2(t,a)\,,\\
	&f(t,x,a)=f_1(t,x)+f_2(t,a)\,
	\end{split}
	\end{equation*}
	for $\forall t\in[0,T]\,,\,\forall x\in\R^d\,,\,\forall a\in A$. In addition, assume that $f$ and $g$ are convex in $x$, {\color{red}$f_2, b_2,\sigma_2$ are convex in $a$}. Then we have the following estimate for the sequence $(\alpha^n)_{n\in\mathbb{N}}$ from Algorithm~\ref{alg mmsa}:
	\begin{equation*}
	0\le J(x,\alpha^n)-J(x,\alpha^*)\le \frac{C}{n}\,,
	\end{equation*}
	where $\alpha^*$ is the optimal control for~\eqref{problem} and $C$ is a positive constant.
\end{corollary}
The proof of Corollary~\ref{cor mmsa rate} will be given in Section~\ref{sec proofs}. Theorem~\ref{thr convergence of modified MSA} and Corollary~\ref{cor mmsa rate} are extensions of the result in~\cite{chernousko} to the stochastic case.

\section{Proofs}\label{sec proofs}
We start working towards the proof of Theorem~\ref{thr convergence of modified MSA}. Recall the adjoint equation for an admissible control $\alpha$:
\begin{equation}\label{eq BSDE}
dY_s^{\alpha}=-D_x\mathcal{H}(s,X_s^{\alpha},Y_s^{\alpha},Z_s^{\alpha},\alpha_s)\,ds+Z_s^{\alpha}\,dW_s,\,s\in[0,T],\,\,Y_T=D_xg(X_T^{\alpha})\,.
\end{equation}
{\color{red}From now on, we shall use Einstein notation, so that repeated indices in a single term imply summation over all the values of that index.}
{\color{red}
\begin{lemma}\label{lem Y Z bounded}
	Assume that there exists $K\ge 0$ 
	such that $\forall x\in\R^d,\forall a\in A,\forall t\in[0,T]$ we have
	\begin{equation*}
	|D_x b(t,x,a)|+|D_x\sigma(t,x,a)|\le K\,,
	\end{equation*}
	and
	\begin{equation*}
	|D_x g(x)|+|D_xf(t,x,a)|\le K\,.
	\end{equation*}
	Then $\|Y^{\alpha}\|_{\mathbb{H}^{\infty}}$  is bounded.
\end{lemma}
\begin{proof}
From the definition of the Hamiltonian~\eqref{eq: hamiltonian} we have
\begin{equation*}
\begin{split}
&D_{x_i}\mathcal{H}(s,X_s^{\alpha},Y_s^{\alpha},Z_s^{\alpha},\alpha_s)=D_{x_i}b^j(s,X_s^{\alpha},\alpha_s) (Y_s^{\alpha})^j+D_{x_i}\sigma^{jp}(s,X_s^{\alpha},\alpha_s) (Z_s^{\alpha})^{jp}\\
&\qquad\qquad\qquad\qquad\qquad\qquad\qquad+D_{x_i}f(s,X_s^{\alpha},\alpha_s)\,,\,\,\forall s\in[0,T]\,,\,\,i=1,2,\dots,d\,.
\end{split}
\end{equation*}
Hence, one can observe that~\eqref{eq BSDE} is a linear BSDE. Therefore, from~\cite[Proposition 3.2]{richou} we can write the formula for the solution of~\eqref{eq BSDE}:
\begin{equation*}
	Y_t^{\alpha}=\E_t\left[S_t^{-1}S_T D_xg(X_T^{\alpha})+\int_t^T S_t^{-1}S_sD_xf(s,X_s^{\alpha},\alpha_s)\,ds\right]\,,
\end{equation*}
where the process $S$ is the unique strong solution of
\begin{equation*}
	dS_t^{ij}=S^{il}_t D_{x_l}b^j(t,X_t^{\alpha},\alpha_t)\,dt+S^{il}_t D_{x_l}\sigma^{jp}(t,X_t^{\alpha},\alpha_t)\,dW_t^p\,,\,i,j=1,2,\dots,d,\,S_0=I_d\,,
\end{equation*}
and $S^{-1}$ is the inverse process of $S$. Thus, due to~\cite[Corollary 3.7]{richou} and assumptions of lemma we have the following bound:
\begin{equation*}
	\|Y^{\alpha}\|_{\mathbb{H}^{\infty}}\le C\|D_xg(X_T^{\alpha})\|_{L^{\infty}}+CT\|D_xf(\cdot,X_{\cdot}^{\alpha},\alpha_{\cdot})\|_{\mathbb{H}^{\infty}}\,.
\end{equation*}
Hence, due to assumptions of lemma we conclude that $\|Y^{\alpha}\|_{\mathbb{H}^{\infty}}$ is bounded.
\end{proof}
}

\begin{proof}[Proof of Lemma~\ref{lem estimate for difference of J}]
	Let $\varphi$ and $\theta$ be some generic admissible controls. We will write $(X^{\varphi}_s)_{s\in[0,T]}$ for the solution of~\eqref{sde} controlled by $\varphi$ and $(X^{\theta}_s)_{s\in[0,T]}$ for the solution of~\eqref{sde} controlled by $\theta$. We denote solutions of corresponding adjoint equations by $(Y^{\varphi}_s,\,Z^{\varphi}_s)_{s\in[0,T]}$ and  $(Y^{\theta}_s,\,Z^{\theta}_s)_{s\in[0,T]}$. Due to Taylor's theorem, we note that for some $R^1(\omega)\in[0,1]$, we have $\forall\omega\in\Omega$ that
	{\color{red}
	\begin{equation*}
	\begin{split}
	g(X^\varphi_T)&-g(X^\theta_T)=  (D_xg(X^\theta_T))^\top(X^\varphi_T-X^\theta_T)\\
	&\qquad+\frac{1}{2}(X^\varphi_T-X^\theta_T)^\top D^2_{x}g(X^\theta_T+R^1(X^\varphi_T-X^\theta_T))(X^\varphi_T-X^\theta_T)\\
	&\le (D_xg(X^\theta_T))^\top(X^\varphi_T-X^\theta_T)\\
	&\qquad+\frac{1}{2}(X^\varphi_T-X^\theta_T)^\top \left|D^2_{x}g(X^\theta_T+R^1(X^\varphi_T-X^\theta_T))\right|(X^\varphi_T-X^\theta_T)\\
	&\le (D_xg(X^\theta_T))^\top(X^\varphi_T-X^\theta_T)+\frac{K}{2}\left|X^\varphi_T-X^\theta_T\right|^2\,.
	\end{split}
	\end{equation*}
	The last inequality holds due to Assumption~\ref{ass second order derivative of Hamiltonian bounded}.
	} 
	Recall that $Y^\theta_T=D_xg(X^\theta_T)$. 
	Hence, using It\^o's product rule, we get 
	{\color{red}
	\begin{equation*}
	\begin{split}
	\E&[g(X^\varphi_T)-g(X^\theta_T)] \le \E \left[(Y^\theta_T)^\top(X^\varphi_T-X^\theta_T)+\frac{K}{2}\left|X^\varphi_T-X^\theta_T\right|^2\right]\\ 
	&\le\E\int_0^T (X^\varphi_s-X^\theta_s)^\top\,dY^\theta_s+\E\int_0^T (Y^\theta_s)^\top[dX^\varphi_s-dX^\theta_s]\\
	&\qquad+\E\int_0^T\text{tr}[(\sigma(s,X^\varphi_s,\varphi_s)-\sigma(s,X^\theta_s,\theta_s))^\top Z^\theta_s]\,ds+\frac{K}{2}\E\left[\left|X^\varphi_T-X^\theta_T\right|^2\right]\,.\\
	\end{split}
	\end{equation*}
	}
	From this, the forward SDE~\eqref{sde} and the adjoint equation~\eqref{eq: adjoint equation} we thus get
	\begin{equation}\label{eq diff of g}
	\begin{split}
	& \E[g(X^\varphi_T)-g(X^\theta_T)] \\
	&\le-\E\int_0^T(X^\varphi_s-X^\theta_s)^\top D_x\mathcal{H}(s,X^\theta_s,Y^\theta_s,Z^\theta_s,\theta_s)\,ds\\
	&\qquad+\E\int_0^T(Y^\theta_s)^\top[b(s,X^\varphi_s,\varphi_s)-b(s,X^\theta_s,\theta_s)]\,ds\\
	&\qquad+\E\int_0^T\text{tr}[(\sigma(s,X^\varphi_s,\varphi_s)-\sigma(s,X^\theta_s,\theta_s))^\top Z^\theta_s]\,ds+\frac{K}{2}\E\left[\left|X^\varphi_T-X^\theta_T\right|^2\right]\,.
	\end{split}
	\end{equation}
	On the other hand, by definition of the Hamiltonian we have
	\begin{equation}\label{eq diff of f}
	\begin{split}
	\E&\int_0^T[f(s,X^\varphi_s,\varphi_s)-f(s,X^\theta_s,\theta_s)]\,ds\\
	&=\E\int_0^T[\mathcal{H}(s,X^\varphi_s,Y^\theta_s,Z^\theta_s,\varphi_s)-\mathcal{H}(s,X^\theta_s,Y^\theta_s,Z^\theta_s,\theta_s)]\,ds\\
	&\qquad-\E\int_0^T(Y^\theta_s)^\top[b(s,X^\varphi_s,\varphi_s)-b(s,X^\theta_s,\theta_s)]\,ds\\
	&\qquad-\E\int_0^T\text{tr}[(\sigma(s,X^\varphi_s,\varphi_s)-\sigma(s,X^\theta_s,\theta_s))^\top Z^\theta_s]\,ds\,.
	\end{split}
	\end{equation}
	Summing up~\eqref{eq diff of g} and~\eqref{eq diff of f} we get
	\begin{equation}\label{eq diff of J}
	\begin{split}
	& J(x,\varphi)-J(x,\theta)\\
	& =\E[g(X^\varphi_T)-g(X^\theta_T)]+\E\int_0^T[f(s,X^\varphi_s,\varphi_s)-f(s,X^\theta_s,\theta_s)]\,ds\\
	&\le \E\int_0^T[\mathcal{H}(s,X^\varphi_s,Y^\theta_s,Z^\theta_s,\varphi_s)-\mathcal{H}(s,X^\theta_s,Y^\theta_s,Z^\theta_s,\theta_s)\\
	&\qquad-(X^\varphi_s-X^\theta_s)^\top D_x\mathcal{H}(s,X^\theta_s,Y^\theta_s,Z^\theta_s,\theta_s)]\,ds+\frac{K}{2}\E\left[\left|X^\varphi_T-X^\theta_T\right|^2 \right]\,.
	\end{split}
	\end{equation}
	Due to Taylor's theorem, there exists   $(R^2_s(\omega))_{s\in[0,T]}\in[0,1]$ such that $\forall\omega\in\Omega$ we have
	{\color{red}
	\begin{equation}\label{eq diff of Hamiltonian}
	\begin{split}
	&\mathcal{H}(s,X^\varphi_s,Y^\theta_s,Z^\theta_s,\varphi_s)- \mathcal{H}(s,X^\theta_s,Y^\theta_s,Z^\theta_s,\theta_s)\\
	&=\mathcal{H}(s,X^\theta_s,Y^\theta_s,Z^\theta_s,\varphi_s)-\mathcal{H}(s,X^\theta_s,Y^\theta_s,Z^\theta_s,\theta_s)\\
	&+(X^\varphi_s-X^\theta_s)^\top D_x\mathcal{H}(s,X^\theta_s,Y^\theta_s,Z^\theta_s,\varphi_s)\\
	&+\frac{1}{2}(X^\varphi_s-X^\theta_s)^\top D_{x}^2\mathcal{H}(s,X^\theta_s+R^2_s(X^\varphi_s-X^\theta_s),Y^\theta_s,Z^\theta_s,\varphi_s)(X^\varphi_s-X^\theta_s)\,.
	\end{split}
	\end{equation}
	Since $D_x^2\sigma(s,X^\theta_s+R^2_s(X^\varphi_s-X^\theta_s),\varphi_s)=0$ by Assumption~\ref{assumption controlled SDE for modified MSA}, we have that
	\begin{equation*}
		\begin{split}
		&\left|D_{x_ix_j}^2\mathcal{H}(s,X^\theta_s+R^2_s(X^\varphi_s-X^\theta_s),Y^\theta_s,Z^\theta_s,\varphi_s)\right|\\
		&\qquad=\left|D_{x_i x_j}^2 b^l(s,X^\theta_s+R^2_s(X^\varphi_s-X^\theta_s),\varphi_s)(Y^\theta_s)^l\right.\\
		&\qquad\qquad\left.+D_{x_i x_j}^2 f(s,X^\theta_s+R^2_s(X^\varphi_s-X^\theta_s),\varphi_s)\right|\,,\,i,j=1,2,\dots,d\,.
		\end{split}
	\end{equation*}
	From Lemma~\ref{lem Y Z bounded} we know that $|Y^\theta_s|$ is bounded a.s. for all $s\in[0,T]$.
	}
	 Hence by Assumption~\ref{assumption controlled SDE for modified MSA} and~\ref{ass second order derivative of Hamiltonian bounded} we have
	\begin{equation}\label{eq D_x^2 H is bounded}
		|D_{x}^2\mathcal{H}(s,X^\theta_s+R^2_s(X^\varphi_s-X^\theta_s),Y^\theta_s,Z^\theta_s,\varphi_s)|<\infty\,.
	\end{equation}
	Therefore, after substituting~\eqref{eq diff of Hamiltonian} into~\eqref{eq diff of J}, and by~\ref{eq D_x^2 H is bounded} we get
	\begin{equation*}
	\begin{split}
	J(x,&\varphi)-J(x,\theta)= \E\left[\int_0^T[\mathcal{H}(s,X^\theta_s,Y^\theta_s,Z^\theta_s,\varphi_s)-\mathcal{H}(s,X^\theta_s,Y^\theta_s,Z^\theta_s,\theta_s)\right.\\
	&+(X^\varphi_s-X^\theta_s)^\top(D_x\mathcal{H}(s,X^\theta_s,Y^\theta_s,Z^\theta_s,\varphi_s)-D_x\mathcal{H}(s,X^\theta_s,Y^\theta_s,Z^\theta_s,\theta_s))\\
	&\left.+\frac{K}{2}\left|X^\varphi_s-X^\theta_s\right|^2\,ds\right]+\frac{K}{2}\E\left[\left|X^\varphi_T-X^\theta_T\right|^2\right]\,.
	\end{split}
	\end{equation*}
	Let us now get a standard SDE estimate for the difference of $X^{\varphi}$ and $X^{\theta}$. From $(a+b)^2\le 2a^2+2b^2$, from taking the expectation, from H\"older's inequality, from Assumption~\ref{assumption controlled SDE for modified MSA}, from the Burkholder-Davis-Gundy inequality and from Gronwall's inequality we obtain
	\begin{equation}\label{eq standatd sde estimate}
	\begin{split}
	\E\sup_{0\le t\le T}&|X^\varphi_t-X^\theta_t|^2\le C\E\int_0^T|b(s,X^\theta_s,\varphi_s)-b(s,X^\theta_s,\theta_s)|^2\,ds\\
	&+C\E\int_0^T|\sigma(s,X^\theta_s,\varphi_s)-\sigma(s,X^\theta_s,\theta_s)|^2\,ds\,.
	\end{split}
	\end{equation}
	Young's inequality allows us to get the estimate
	\begin{equation*}
	\begin{split}
	J&(x,\varphi)-J(x,\theta)\\
	&\le\E\int_0^T[\mathcal{H}(s,X^\theta_s,Y^\theta_s,Z^\theta_s,\varphi_s)-\mathcal{H}(s,X^\theta_s,Y^\theta_s,Z^\theta_s,\theta_s)]\,ds+\frac{1}{2}\E\int_0^T|X^\varphi_s-X^\theta_s|^2\,ds\\
	&\quad\qquad+\frac{1}{2}\E\left[\int_0^T|D_x\mathcal{H}(s,X^\theta_s,Y^\theta_s,Z^\theta_s,\varphi_s)-D_x\mathcal{H}(s,X^\theta_s,Y^\theta_s,Z^\theta_s,\theta_s)|^2\right.\\
	&\quad\qquad
	\qquad\left.+\frac{K}{2}\left|X^\varphi_s-X^\theta_s\right|^2\,ds\right]+\frac{K}{2}\E\left[\left|X^\varphi_T-X^\theta_T\right|^2\right]\,.
	\end{split}
	\end{equation*}
	Hence, from~\eqref{eq standatd sde estimate} we have that
	\begin{equation*}
	\begin{split}
	J&(x,\varphi)-J(x,\theta) \le\E\int_0^T[\mathcal{H}(s,X^\theta_s,Y^\theta_s,Z^\theta_s,\varphi_s)-\mathcal{H}(s,X^\theta_s,Y^\theta_s,Z^\theta_s,\theta_s)]\,ds\\
	&\quad+C\E\int_0^T|b(s,X^\theta_s,\varphi_s)-b(s,X^\theta_s,\theta_s)|^2\,ds\\
	&\quad+C\E\int_0^T|\sigma(s,X^\theta_s,\varphi_s)-\sigma(s,X^\theta_s,\theta_s)|^2\,ds\\
	&\quad+C\E\int_0^T|D_x\mathcal{H}(s,X^\theta_s,Y^\theta_s,Z^\theta_s,\varphi_s)-D_x\mathcal{H}(s,X^\theta_s,Y^\theta_s,Z^\theta_s,\theta_s)|^2\,ds\,,
	\end{split}
	\end{equation*}
	for some constant $C>0$, which depends on $K,T$, and $d$.
\end{proof}

\begin{proof}[Proof of Theorem~\ref{thr extended pop}]
	Since $\alpha^*$ is the {\color{red}(locally)} optimal control for the problem~\eqref{problem}, the Pontryagin's optimality principle holds, see e.g.~\cite{pham book}. Hence for any $a\in A$ we have
	\begin{equation}\label{eq Pontryagin principle}
	\mathcal{H}(s,X_s^{\alpha^*},Y_s^{\alpha^*},Z_s^{\alpha^*},\alpha_s^*)\le \mathcal{H}(s,X_s^{\alpha^*},Y_s^{\alpha^*},Z_s^{\alpha^*},a)\,,\,\,\,\forall s\in[0,T]\,.
	\end{equation}
	By definition of the augmented Hamiltonian~\eqref{eq: augmented hamiltonian} for all $s\in[0,T]$ we have
	\begin{equation}\label{eq recall augmented hamiltonian}
	\begin{split}
	&\tilde{\mathcal{H}}(s,X_s^{\alpha^*},Y_s^{\alpha^*},Z_s^{\alpha^*},\alpha_s^*,a)=\mathcal{H}(s,X_s^{\alpha^*},Y_s^{\alpha^*},Z_s^{\alpha^*},a)\\
	&\qquad+\frac{1}{2}\rho|b(s,X_s^{\alpha^*},a)-b(s,X_s^{\alpha^*},\alpha_s^*)|^2+\frac{1}{2}\rho|\sigma(s,X_s^{\alpha^*},a)-\sigma(s,X_s^{\alpha^*},\alpha_s^*)|^2\\
	&\qquad+\frac{1}{2}\rho|D_x\mathcal{H}(s,X_s^{\alpha^*},Y_s^{\alpha^*},Z_s^{\alpha^*},a)-D_x\mathcal{H}(s,X_s^{\alpha^*},Y_s^{\alpha^*},Z_s^{\alpha^*},\alpha_s^*)|^2\,.
	\end{split}
	\end{equation}
	Therefore, due to~\eqref{eq Pontryagin principle} and~\eqref{eq recall augmented hamiltonian} we have
	\begin{equation*}
	\begin{split}
	\tilde{\mathcal{H}}&(s,X_s^{\alpha^*},Y_s^{\alpha^*},Z_s^{\alpha^*},\alpha_s^*,\alpha_s^*)=\mathcal{H}(s,X_s^{\alpha^*},Y_s^{\alpha^*},Z_s^{\alpha^*},\alpha_s^*)\\
	&\le \mathcal{H}(s,X_s^{\alpha^*},Y_s^{\alpha^*},Z_s^{\alpha^*},a)+\frac{1}{2}\rho|b(s,X_s^{\alpha^*},a)-b(s,X_s^{\alpha^*},\alpha_s^*)|^2\\
	&\qquad+\frac{1}{2}\rho|\sigma(s,X_s^{\alpha^*},a)-\sigma(s,X_s^{\alpha^*},\alpha_s^*)|^2\\
	&\qquad+\frac{1}{2}\rho|D_x\mathcal{H}(s,X_s^{\alpha^*},Y_s^{\alpha^*},Z_s^{\alpha^*},a)-D_x\mathcal{H}(s,X_s^{\alpha^*},Y_s^{\alpha^*},Z_s^{\alpha^*},\alpha_s^*)|^2\\
	&=\tilde{\mathcal{H}}(s,X_s^{\alpha^*},Y_s^{\alpha^*},Z_s^{\alpha^*},\alpha_s^*,a)\,.
	\end{split}
	\end{equation*}
	This concludes the proof.
\end{proof}

\begin{proof}[Proof of Theorem~\ref{thr convergence of modified MSA}]
	Let us apply Lemma~\ref{lem estimate for difference of J} for $\varphi=\alpha^{n}$ and $\theta=\alpha^{n-1}$. Hence, for some $C>0$ we have
	\begin{equation}\label{eq difference of J in the proof}
	\begin{split}
	J&(x,\alpha^{n})-J(x,\alpha^{n-1})\\
	&\le\E\int_0^T[\mathcal{H}(s,X^{n}_s,Y^{n}_s,Z^{n}_s,\alpha^{n}_s)-\mathcal{H}(s,X^{n}_s,Y^{n}_s,Z^{n}_s,\alpha^{n-1}_s)]\,ds\\
	&\quad+C\E\int_0^T|b(s,X^{n}_s,\alpha^{n}_s)-b(s,X^{n}_s,\alpha^{n-1}_s)|^2\,ds\\
	&\quad+C\E\int_0^T|\sigma(s,X^{n}_s,\alpha^{n}_s)-\sigma(s,X^{n}_s,\alpha^{n-1}_s)|^2\,ds\\
	&\quad+C\E\int_0^T\left|D_x\mathcal{H}(s,X^{n}_s,Y^{n}_s,Z^{n}_s,\alpha^{n}_s)-D_x\mathcal{H}(s,X^{n}_s,Y^{n}_s,Z^{n}_s,\alpha^{n-1}_s)\right|^2\,ds\,.
	\end{split}
	\end{equation}
	Let
	\begin{equation*}
	\mu(\alpha^{n-1})=\E\int_0^T[\mathcal{H}(s,X^{n}_s,Y^{n}_s,Z^{n}_s,\alpha^{n}_s)-\mathcal{H}(s,X^{n}_s,Y^{n}_s,Z^{n}_s,\alpha^{n-1}_s)]\,ds\,.
	\end{equation*}
	Due to the definition of $\alpha^n$~\eqref{eq def anplus1 MMSA} and~\eqref{eq extended pontryagin} we have for all $s\in[0,T]$
	\begin{equation*}
	\begin{split}
	\mathcal{H}&(s,X_s^{n},Y_s^{n},Z_s^{n},\alpha^{n}_s)+\frac{1}{2}\rho|b(s,X^{n}_s,\alpha^{n}_s)-b(s,X^{n}_s,\alpha^{n-1}_s)|^2\\
	&\qquad+\frac{1}{2}\rho|\sigma(s,X^{n}_s,\alpha^{n}_s)-\sigma(s,X^{n}_s,\alpha^{n-1}_s)|^2\\
	&\qquad+\frac{1}{2}\rho|D_x\mathcal{H}(s,X^{n}_s,Y^{n}_s,Z^{n}_s,\alpha^{n}_s)-D_x\mathcal{H}(s,X^{n}_s,Y^{n}_s,Z^{n}_s,\alpha^{n-1}_s)|^2\\
	&\le \mathcal{H}(s,X_s^{n},Y_s^{n},Z_s^{n},\alpha^{n-1}_s)\,.
	\end{split}
	\end{equation*}
	Therefore, we can observe that $\mu(\alpha^{n-1})\le0$.
	Hence we can rewrite the inequality~\eqref{eq difference of J in the proof} as
	\begin{equation}\label{eq J(alpha n)-J(alpha n-1)}
	\begin{split}
	J(x,\alpha^{n})&-J(x,\alpha^{n-1})\le \mu(\alpha^{n-1})-\frac{2C}{\rho}\mu(\alpha^{n-1})=D\mu(\alpha^{n-1})\,,
	\end{split}
	\end{equation}
	where $D:=1-\frac{2C}{\rho}$. 
	By choosing $\rho>2C$ we have that $D>0$.  
	Notice that for any integer $M>1$ we have 
	\begin{equation*}
	\begin{split}
	\sum_{n=1}^M&(-\mu(\alpha^{n-1}))\le D^{-1}\sum_{n=1}^M(J(x,\alpha^{n-1})-J(x,\alpha^{n}))\\
	&\,=D^{-1}(J(x,\alpha^0)-J(x,\alpha^{M}))\le D^{-1}(J(x,\alpha^{0})-\inf_{\alpha\in\mathcal{A}}J(x,\alpha))<\infty.
	\end{split}
	\end{equation*}
	Since $(-\mu(\alpha^{n-1}))\ge 0$ and $\sum_{n=1}^\infty(-\mu(\alpha^{n-1}))<+\infty$ we have that $\mu(\alpha^{n-1})\rightarrow 0$ as $n\rightarrow 0$. This concludes the proof.
\end{proof}
We need to introduce new notation, which will be used in the proof of Corollary~\ref{cor mmsa rate}. {\color{red}Denote the set
\begin{equation}\label{eq def of I}
	I_{\tau, h}:=[\tau-h,\tau+h]\cap[0,T]\,\,,\,\tau\in[0,T],\,h\in[0,+\infty)\,.
\end{equation} 
}
Let us define for all $s\in[0,T]$
\begin{equation*} 
\Delta_{\alpha^{n-1}} \mathcal{H}(s):=\mathcal{H}(s,X^{n}_s,Y^{n}_s,Z^{n}_s,\alpha^{n}_s)-\mathcal{H}(s,X^{n}_s,Y^{n}_s,Z^{n}_s,\alpha^{n-1}_s)\,,
\end{equation*} 
and
\begin{equation*}
\mu(\alpha^{n-1}):=\E\int_{0}^T\Delta_{\alpha^{n-1}} \mathcal{H}(s)\,ds\,.
\end{equation*}
By definition of $\alpha^n$ notice that $\Delta_{\alpha^{n-1}} \mathcal{H}(t)\le0$ for all $t\in[0,T]$. Let us show an auxiliary lemma.
\begin{lemma}\label{lem delta H E tau h}
	For any $h>0$ there exists $\tau$, which depends on $h$ and $\alpha^{n-1}$, such that
	\begin{equation*}
	\E\int_{I_{\tau,h}}\Delta_{\alpha^{n-1}} \mathcal{H}(t)\,dt\le \frac{h\mu(\alpha^{n-1})}{T}\,.
	\end{equation*}
\end{lemma}
\begin{proof}
	We will prove by contradiction. Assume that there exists $h^*>0$ such that $\forall\tau\in[0,T]$ we have
	\begin{equation}\label{eq assumption for contradiction}
	\E\int_{I_{\tau, h^*}}\Delta_{\alpha^{n-1}} \mathcal{H}(t)\,dt>\frac{h^*\mu(\alpha^{n-1})}{T}\,.
	\end{equation}
	Denote $\tau_i=ih^*$, $i=0,1,\dots,N(h^*)$, where $N(h^*)=[T/h^*]$ - integer part.
	{\color{red}Since $\Delta_{\alpha^{n-1}} \mathcal{H}(t)\le 0$ for all $t\in[0,T]$ by definition of $\alpha^n$ and $\cup_{i=0}^{N(h^*)}I_{\tau_i, h^*}$ is a superset of $[0,T]$ we have 
	\begin{equation}
	\begin{split}
	\mu(\alpha^{n-1})=\E\int_{0}^T\Delta_{\alpha^{n-1}} \mathcal{H}(t)\,dt\ge \sum_{i=0}^{N(h^*)}\E\int_{I_{\tau_i, h^*}}\Delta_{\alpha^{n-1}} \mathcal{H}(t)\,dt\,.
	\end{split}
	\end{equation}
Hence, by~\eqref{eq assumption for contradiction} we get
	\begin{equation*}
	\mu(\alpha^{n-1})>\frac{h^* N(h^*)}{T}\E\int_0^T\Delta_{\alpha^{n-1}} \mathcal{H}(t)\,dt>\mu(\alpha^{n-1})\,.
	\end{equation*}
	}
	Hence we get the contradiction.
\end{proof}
Now we are ready to prove Corollary~\ref{cor mmsa rate}.
\begin{proof}[Proof of Corollary~\ref{cor mmsa rate}]
First, observe that
\begin{equation*}
\begin{split}
&b(s,X^n_s,\alpha^{n}_s)-b(s,X^n_s,\alpha^{n-1}_s)=b_2(s,\alpha^{n}_s)-b_2(s,\alpha^{n-1}_s)\,,\\
&\sigma(s,X^n_s,\alpha^{n}_s)-\sigma(s,X^n_s,\alpha^{n-1}_s)=\sigma_2(s,\alpha^{n}_s)-\sigma_2(s,\alpha^{n-1}_s)\,,\\
&D_x\mathcal{H}(s,X^n_s,Y^n_s,Z^n_s,\alpha^{n}_s)-D_x\mathcal{H}(s,X^n_s,Y^n_s,Z^n_s,\alpha^{n-1}_s)=0\,.
\end{split}
\end{equation*}
{\color{red}Let us consider the set $I_{\tau, h}$ given by~\eqref{eq def of I}. We will specify the choice of $\tau$ and $h$ later.}  Hence, after applying Lemma~\ref{lem estimate for difference of J} for $\alpha^n$ and $\alpha^{n-1}$ we have for some $C>0$
\begin{equation*}
\begin{split}
J&(x,\alpha^{n})-J(x,\alpha^{n-1})\\
\le & \E\int_{I_{\tau, h}}[\mathcal{H}(s,X^n_s,Y^n_s,Z^n_s,\alpha^{n}_s)-\mathcal{H}(s,X^n_s,Y^n_s,Z^n_s,\alpha^{n-1}_s)]\,ds\\
&+C\E\int_{I_{\tau, h}}|b_2(s,\alpha^{n}_s)-b_2(s,\alpha^{n-1}_s)|^2+|\sigma_2(s,\alpha^{n}_s)-\sigma_2(s,\alpha^{n-1}_s)|^2\,ds\\
&+ \E\int_{[0,T]\setminus I_{\tau, h}}[\mathcal{H}(s,X^n_s,Y^n_s,Z^n_s,\alpha^{n}_s)-\mathcal{H}(s,X^n_s,Y^n_s,Z^n_s,\alpha^{n-1}_s)]\,ds\\
&+C\E\int_{[0,T]\setminus I_{\tau, h}}|b_2(s,\alpha^{n}_s)-b_2(s,\alpha^{n-1}_s)|^2+|\sigma_2(s,\alpha^{n}_s)-\sigma_2(s,\alpha^{n-1}_s)|^2\,ds\,.
\end{split}
\end{equation*}
{\color{red}Since the following holds for all $s\in[0,T]$ and $\rho\ge0$:
\begin{equation*}
\begin{split}
\mathcal{H}&(s,X_s^{n},Y_s^{n},Z_s^{n},\alpha^{n}_s)-\mathcal{H}(s,X_s^{n},Y_s^{n},Z_s^{n},\alpha^{n-1}_s)\\
&+\frac{1}{2}\rho|b_2(s,\alpha^{n}_s)-b_2(s,\alpha^{n-1}_s)|^2+\frac{1}{2}\rho |\sigma_2(s,\alpha^{n}_s)-\sigma_2(s,\alpha^{n-1}_s)|^2\le 0 \,,
\end{split}
\end{equation*}
}
we have for $\rho\ge 2C$
\begin{equation*}
\begin{split}
J&(x,\alpha^{n})-J(x,\alpha^{n-1})\\
&\le  \E\int_{I_{\tau, h}}[\mathcal{H}(s,X^n_s,Y^n_s,Z^n_s,\alpha^{n}_s)-\mathcal{H}(s,X^n_s,Y^n_s,Z^n_s,\alpha^{n-1}_s)]\,ds\\
&\qquad+C\E\int_{I_{\tau, h}}|b_2(s,\alpha^{n}_s)-b_2(s,\alpha^{n-1}_s)|^2+|\sigma_2(s,\alpha^{n}_s)-\sigma_2(s,\alpha^{n-1}_s)|^2\,ds\,.
\end{split}
\end{equation*}
Therefore, from Lemma~\ref{lem delta H E tau h} and from similar calculations as in~\eqref{eq J(alpha n)-J(alpha n-1)},  {\color{red}there exists $\tau$ such that}
\begin{equation*}
\begin{split}
J&(x,\alpha^{n})-J(x,\alpha^{n-1})\\
&\le\left(1-\frac{2C}{\rho}\right)\E\int_{I_{\tau,  h}}[\mathcal{H}(s,X^n_s,Y^n_s,Z^n_s,\alpha^{n}_s)-\mathcal{H}(s,X^n_s,Y^n_s,Z^n_s,\alpha^{n-1}_s)]\,ds\\
&\le \left(1-\frac{2C}{\rho}\right)\frac{h\mu(\alpha^{n-1})}{T}\,.
\end{split}
\end{equation*}
Let us choose $h=-(\rho-2C)\mu(\alpha^{n-1})/(\rho T)$. Hence
\begin{equation}\label{eq diff of J n and n-1}
J(x,\alpha^{n})-J(x,\alpha^{n-1})\le -(\rho-2C)^2(\mu(\alpha^{n-1}))^2/(\rho^2 T^2).
\end{equation}
Let $\alpha^*$ be the optimal control. {\color{red} Indeed, by the sufficient condition for optimality, see e.g.~\cite{pham book}, and by assumptions of corollary, we have the existence of the optimal control.} Therefore, by convexity of $g$, and by It\^o's product rule we have
\begin{equation*}
\begin{split}
0\le & J(x,\alpha^{n-1})-J(x,\alpha^*)\\
= & \E\left[\int_0^T(f(s,X^{n}_s,\alpha^{n-1}_s)-f(s,X_s,\alpha_s^*))\,ds+g(X^{n}_T)-g(X_T)\right]\\
\le & \E\left[\int_0^T(f(s,X^{n}_s,\alpha^{n-1}_s)-f(s,X_s,\alpha_s^*))\,ds\right]+\E[(D_xg(X^{n}))^
\top(X^{n}_T-X_T)]\\
\le & \E\left[\int_0^T(f(s,X^{n}_s,\alpha^{n-1}_s)-f(s,X_s,\alpha_s^*))\,ds\right]\\
&\qquad+  \E\left[\int_0^T (Y^{n}_s)^\top d(X^{n}_s-X_s)+\int_0^T(X^{n}_s-X_s)^\top dY^{n}_s\right]\\
&\qquad+  \E\left[\int_0^T\text{tr}((\sigma(s,X^{n}_s,\alpha^{n-1}_s)-\sigma(s,X_s,\alpha_s^*))^\top Z^{n}_s)\,ds\right]\,.
\end{split}
\end{equation*}
Hence, we have that
\begin{equation*}
\begin{split}
0\le & J(x,\alpha^{n-1})-J(x,\alpha^*)
\le \E\left[\int_0^Tf(s,X^{n}_s,\alpha^{n-1}_s)-f(s,X_s,\alpha_s^*)\,ds\right]\\
&+\E\left[\int_0^T (Y^{n}_s)^\top (b(s,X^{n}_s,\alpha^{n-1}_s)-b(s,X_s,\alpha_s^*))\,ds\right]\\
&-\E\left[\int_0^T(X^{n}_s-X_s)^\top D_x\mathcal{H}(s,X^{n}_s,Y^{n}_s,Z^{n}_s,\alpha^{n-1}_s)\,ds\right]\\
&+\E\left[\int_0^T\text{tr}((\sigma(s,X^{n}_s,\alpha^{n-1}_s)-\sigma(s,X_s,\alpha_s^*))^\top Z^{n}_s)\,ds\right]\,.
\end{split}
\end{equation*}
Recalling the form of $b,\sigma$ and observing that
\begin{equation*}
D_x\mathcal{H}(s,X^{n}_s,Y^{n}_s,Z^{n}_s,\alpha^{n-1}_s)=b_1(s)Y^n_s+\sigma_1(s)Z_s^n+D_xf(s,X^n_s,\alpha^{n-1}_s)\,,
\end{equation*}
we have
\begin{equation*}
\begin{split}
0\le &J(x,\alpha^{n-1})-J(x,\alpha^*)
\le \E\left[\int_0^Tf(s,X^{n}_s,\alpha^{n-1}_s)-f(s,X_s,\alpha_s^*)\,ds\right]\\
&+\E\left[\int_0^T\text{tr}((\sigma_2(s,\alpha^{n-1}_s)-\sigma_2(s,\alpha_s^*))^\top Z^{n}_s)\,ds\right]\\
&+\E\left[\int_0^T (Y^{n}_s)^\top(b_2(s,\alpha^{n-1}_s)-b_2(s,\alpha_s^*))\,ds-\int_0^T(X^{n}_s-X_s)^\top D_xf(s,X^{n}_s,\alpha^{n-1}_s)\,ds\right]\,.
\end{split}
\end{equation*}
Since $f$ is convex in $x$ we have for all $s\in[0,T]$ that
\begin{equation*}
f(s,X_s,\alpha^{n-1}_s)\ge f(s,X_s^{n},\alpha_s^{n-1})+(X_s-X^{n}_s)^\top D_xf(s,X^{n}_s,\alpha^{n-1}_s)\,.
\end{equation*}
Therefore, we obtain
\begin{equation}\label{eq diff of J n-1 and alpha}
\begin{split}
J(&x,\alpha^{n-1})-J(x,\alpha^*)\\
&\le \E\int_0^T\left[\mathcal{H}(s,X^{n}_s, Y^{n}_s,Z^{n}_s,\alpha^{n-1}_s)-\mathcal{H}(s,X^{n}_s, Y^{n}_s,Z^{n}_s,\alpha_s^*)\right]\,ds\\
&\le-\mu(\alpha^{n-1})\,,
\end{split}
\end{equation}
where the second inequality holds due to 
\begin{equation*}
\mathcal{H}(s,X^{n}_s, Y^{n}_s,Z^{n}_s,\alpha^{n}_s)\le \mathcal{H}(s,X^{n}_s, Y^{n}_s,Z^{n}_s,\alpha_s^*)\,.
\end{equation*}
Let $b^n:=J(x,\alpha^n)-J(x,\alpha)$, then due to~\eqref{eq diff of J n and n-1} and~\eqref{eq diff of J n-1 and alpha} we have that
\begin{equation*}
b^{n}-b^{n-1}\le \frac{-(\rho-2C)^2\mu(\alpha^{n-1})^2}{(\rho^2T^2)}\le \frac{-(\rho-2C)^2(b^{n-1})^2}{\rho^2T^2}\,.
\end{equation*}
Therefore, due to Lemma~\ref{lem sequence b_k} we have
\begin{equation*}
J(x,\alpha^n)-J(x,\alpha^*)\le \frac{C_1}{n}\,.
\end{equation*}
 for some constant $C_1>0$. This concludes the proof.
\end{proof}

\appendix
\section{Auxiliary Lemma}\label{sec appendix}

\begin{lemma}\label{lem sequence b_k}
	Let $\{b_k\}_{k\in\mathbb{N}}$ be the sequence of nonnegative numbers such that
	\begin{equation*}
	b_{k+1}\le b_k-q b_k^2\,,
	\end{equation*}
	where $q$ is a positive constant. Then $b_k=O(1/k)$.
\end{lemma}
One can find the proof in~\cite[Lemma 1.4, p. 93]{demyanov book}. However, the proof is written in Russian. For convenience of the reader we provide it here.
\begin{proof}
	Let $b_k=\frac{c_k}{k}$ for some nonnegative sequence $(c_k)_{k\in\mathbb{N}}$. Then it is enough to show that $c_k$ is bounded for all $k\in\mathbb{N}$. By assumption we have
	\begin{equation*}
	b_k-b_{k+1}=\frac{c_k}{k}-\frac{c_{k+1}}{k+1}=\frac{c_k}{k}\left(1-\frac{c_{k+1}}{c_k}\frac{k}{k+1}\right)\ge q\frac{c_k^2}{k^2}\,.
	\end{equation*}
	Therefore,
	\begin{equation*}
	1-\frac{c_{k+1}}{c_k}\frac{k}{k+1}\ge q\frac{c_k}{k}\,.
	\end{equation*}
	After some transformation, we can rewrite the equation above as
	\begin{equation*}
	\left(1+\frac{1}{k}\right)\left(1-q\frac{c_k}{k}\right)\ge \frac{c_{k+1}}{c_k}\,.
	\end{equation*}
	Thus
	\begin{equation*}
	1+\frac{1}{k}(1-qc_k)-q\frac{c_k}{k^2}\ge \frac{c_{k+1}}{c_k}\,.
	\end{equation*}
	If $1-qc_k<0$ we have
	\begin{equation*}
	1>1+\frac{1}{k}(1-qc_k)-q\frac{c_k}{k^2}\ge \frac{c_{k+1}}{c_k}\,.
	\end{equation*}
	Hence $c_{k+1}<c_k$. On the other hand, if $1-qc_k\ge0$, we have $c_k\le\frac{1}{q}$. Therefore, we conclude that for all $k$ we have
	\begin{equation*}
	c_k\le \max\left\{c_1,\frac{1}{q}\right\}\,.
	\end{equation*}
\end{proof}

\end{document}